\documentclass[11pt]{article}

\makeindex
\usepackage[all]{xy}
\usepackage{graphicx}

\usepackage{amsmath}
\usepackage{amsfonts}
\usepackage{amssymb}
\usepackage{amsthm}

\newcommand{\isom}{\cong}

\newcommand{\Z}{{\bf{Z}}}

\newcommand{\Q}{{\bf{Q}}}

\newcommand{\R}{{\bf{R}}}
\newcommand{\C}{{\bf{C}}}

\newcommand{\T}{{\bf{T}}}

\newcommand{\m}{{\mathfrak{m}}}

\newcommand{\Mid}{|} 

\newcommand{\miD}{|}

\newcommand{\divs}{\!\mid\!}
\newcommand{\ndiv}{\!\nmid\!}
\newcommand{\tensor}{\otimes}

\newcommand{\ra}{{\rightarrow}}

 1
\DeclareFontEncoding{OT2}{}{} 
  \newcommand{\textcyr}[1]{%
    {\fontencoding{OT2}\fontfamily{wncyr}\fontseries{m}\fontshape{n}%
     \selectfont #1}}
\newcommand{\Sha}{{\mbox{\textcyr{Sh}}}} 

\newcommand{\pp}{{\mathfrak{p}}}

\newcommand{\OO}{{\mathcal{O}}}
\newcommand{\Of}{{\OO_{\scriptscriptstyle{f}}}}

\usepackage{bm}
\bmdefine\bmu{\mu}

\newcommand{\comment}[1]{}

\DeclareMathOperator{\tor}{tor}

\newtheorem{lem}{Lemma}[section]
\newtheorem{cor}[lem]{Corollary}
\newtheorem{prop}[lem]{Proposition}
\newtheorem{conj}[lem]{Conjecture}
\newtheorem{thm}[lem]{Theorem}

\theoremstyle{definition}

\newtheorem{rmk}[lem]{Remark}

\newcommand{\thetitle}
{Reducibility and rational torsion in modular abelian varieties}

\begin{document}
\parindent=2em

\title{\thetitle}
\author{Amod Agashe and Matthew Winters}

\date{}
\maketitle

\begin{abstract}
Let $N$ be a square-free positive integer and let
$f$ be a newform of weight~$2$ on~$\Gamma_0(N)$. 
Let $A$ denote the abelian subvariety of~$J_0(N)$ associated to~$f$ and
let ${\m}$ be a maximal ideal of the Hecke algebra that contains 
${\rm Ann}_{\T} f$ and has residue characteristic~$r$ such that
$r \ndiv 6N$. 
We show that if either $A[{\m}]$ or
the canonical representation~$\rho_{{\m}}$ over~$\T/\m$ associated to~$\m$
is reducible, then $r$ divides the order of the cuspidal subgroup of~$J_0(N)$
and $A[{\m}]$ has a nontrivial rational point. We mention some 
applications
of this result, including an application 
to the second part of the Birch and Swinnerton-Dyer 
conjecture for~$A$.
\end{abstract}

\section{Introduction} \label{intro}
Let $N$ be a positive integer. Let
$X_0(N)$ denote the modular curve over~$\Q$ associated to~$\Gamma_0(N)$,
and let $J_0(N)$ denote its Jacobian, which is an abelian variety
over~$\Q$.
Let $f$ be a newform of weight~$2$ on~$\Gamma_0(N)$. 
Let $\T$ denote the subring of endomorphisms of~$J_0(N)$
generated by the Hecke operators (usually denoted~$T_\ell$
for $\ell \ndiv N$ and $U_p$ for $p\divs N$). 
Let $I_f = {\rm Ann}_{\T} f$ and let 
$A$ denote the abelian subvariety of~$J_0(N)$ associated to~$f$
(see~\cite[Theorem~7.14]{shimura:intro}).

If $f$ has integer Fourier coefficients, then $A$ is just an elliptic curve,
and it is said to be {\em optimal}; 
suppose
that is the case for the moment. 
Let $r$ be a prime number.
If $A$ has a nontrivial rational point of order~$r$, then clearly
$A[r]$ is reducible
as a representation of~${\rm Gal}(\overline{\Q}/\Q)$. The converse
need not be true.
For example,
for $A = 99d1$ (the notation is as in~\cite{cremona:algs}), $A[5]$ is reducible, but $A$ has no
nontrivial rational $5$-torsion. 
However, the first author had conjectured:

\begin{conj}\cite[Conjecture~2.5]{conjs} \label{conj1}
Suppose $N$ is square-free, i.e., $A$ is semistable, 
and $r$ is an odd prime. 
If $A[r]$ is
reducible, then $A$ has a nontrivial rational~$r$-torsion point 
(recall that we are assuming that $A$ is optimal).
\end{conj}

Note that if $A[2]$ is reducible, then being two dimensional, 
it has a one dimensional invariant subspace, which has to be of the form 
$\{O,P\}$ for some point~$P$; 
the Galois group has to preserve~$O$ and any element $\sigma$ of it cannot move~$P$ to~$O$, 
since its inverse would move~$O$ to~$P$, and thus $\sigma$ has to preserve~$P$, 
making~$P$ a rational point. So the hypothesis that~$r$ 
is odd could have been dropped in the conjecture above, but is 
relevant for a suitable generalization of the conjecture to more general 
abelian varieties associated to newforms (where the argument above does not work).

The conjecture above follows 
from~\cite{vatsal:mult}, though it is not stated as a result there,
since loc. cit. is primarily concerned with other topics.
For the benefit of the literature, we state it as a theorem:

\begin{thm}[Vatsal] \label{vatsal}
If $A$ is an optimal 
semistable elliptic curve
and $r$ is a prime such that 
$A[r]$ is reducible, then $A$ has a nontrivial rational $r$-torsion point.
\end{thm}
\begin{proof}
If $r$ is odd, then by Lemma~5.2 of~\cite{vatsal:mult}, $A$ has ordinary reduction at~$r$; using
this, Proposition~5.3(ii) of~\cite{vatsal:mult} tells us that 
$A$ has a nontrivial rational~$r$-torsion point.
If $r=2$, then as argued above, the conclusion holds (even without the 
optimality or semistability
hypotheses).
\end{proof}

We remark that the optimality hypothesis cannot be removed in the theorem above: 
the elliptic curve $A = 26b2$
has no nontrivial rational $7$-torsion, while $A[7]$ is reducible; here $A$ is not optimal.

In~\cite{conjs}, the author had  remarked that
it was hoped that Conjecture~\ref{conj1} above
``might be  appropriately generalized
to quotients or subvarieties of arbitrary dimension of~$J_0(N)$ associated to newforms'';
at that time, lacking numerical evidence beyond the elliptic curve case, the author
did not make such a generalization.
Let us revert to the case where $f$ need not have integer Fourier coefficients,
so that $A$ need not be an elliptic curve.
It is not immediately clear to the authors of this article if the techniques of~\cite{vatsal:mult}
generalize to this case, since it is a rather lenghty article, whose main goal
is something else. It would be interesting to try to generalize everything in 
that article beyond the elliptic curve case, which would be a bigger project;
in any case, it is not clear a priori what would generalize.

Let $C$ denote the cuspidal subgroup of~$J_0(N)$. If
$N$ is square-free, then all the points of~$C$ are rational
(e.g., by~\cite[Prop.~2]{ogg}).
Recall that to a maximal ideal~$\m$ of the Hecke algebra, there is associated
a canonical Galois representation~$\rho_{\m}: {\rm Gal}(\overline{\Q}/\Q)
\ra {\rm GL}_2(\T/\m)$
(see~\cite[Prop.~5.1]{ribet:modreps}). 
In this article, we prove:
\begin{thm}\label{thm:main}
Suppose $N$ is square-free and ${\m}$ is a maximal ideal of~$\T$ 
that contains~$I_f$ and has 
residue characteristic~$r$ such that $r \ndiv 6N$.
If either $\rho_{\m}$ or~$A[\m]$ is reducible as a representation over~$\T/\m$, then
$A[\m] \cap C$ is nontrivial, and so\\
(a) $r$ divides $|C|$, \\
(b) $A[\m]$ has a nontrivial rational point, and \\
(c) $A$ has a nontrivial rational $r$-torsion point. \\
In particular, if 
$A$ does not have any nontrivial rational $r$-torsion point
or $r$ does not divide $|C|$, then 
$\rho_{\m}$ is irreducible. 
\end{thm}

The proof of the theorem above is given in Section~\ref{sec:lemmas}.
The idea of the proof is very similar 
to that of the main theorem of~\cite{tor}: we use the hypotheses to 
show that the newform~$f$ 
is congruent to an Eisenstein series~$E$
modulo a prime ideal~$\pp$ in the number field generated by the 
Fourier coefficients of~$f$ that lies over~$r$ 
(the tricky part is to prove the congruences for Fourier coefficients
of indices that are not coprime to~$N$). 
As part of the proof of this congruence,
we show that under the hypotheses of the theorem,
for at least one prime~$p$ that divides~$N$,
the sign of the Atkin-Lehner involution at~$p$ acting on~$f$ is~$-1$, which 
is an interesting result on its own
(see Proposition~\ref{prop:mazur2}; some versions of this result have
already appeared in the literature: see Remark~\ref{rmk2}).
Given the congruence between~$f$ and~$E$, and the fact that
$f$ is ordinary at~$\pp$ (which we show), a result of 
Tang~\cite[Thm~0.4]{tang:congruences}
tells us that 
$A[{\m}]$ has nontrivial
intersection with 
a subgroup of the cuspidal group~$C$ of~$J_0(N)$, giving us the theorem. 

The hypothesis that $\rho_{\m}$ is irreducible is often used in arithmetic geometry
(for example in multiplicity one results, as in~\cite[Prop.~1.3]{congnum}). 
One significance of the theorem above is that
the condition that 
$A$ does not have any nontrivial rational $r$-torsion point
may  be easier to check
theoretically or in numerical examples, 
than whether $\rho_{\m}$ is irreducible. 
Another significance of the theorem stems from the fact that the cuspidal subgroup
can be computed (see, e.g.~\cite{stein:cuspidal}), and the 
structure of~$C$ when $N$ is square-free
is given in~\cite{takagi:cuspidal} (see also~\cite{yoo:rational}; the formulas
are too complicated to mention here). As an 
example, the order of the cuspidal subgroup for $N=14$ is~$6$, so for any
maximal ideal~$\m$ of~$\T$ (of level~$14$) of residue characteristic bigger than~$7$, 
$\rho_{\m}$ is irreducible. 

In the next section, we mention 
some more applications of Theorems~\ref{vatsal} and~\ref{thm:main}, 
including one relevant to the second part of the Birch and Swinnerton-Dyer conjecture.
In Section~\ref{sec:lemmas}, we prove some results regarding the 
Fourier coefficients of~$f$ that are needed to show the congruence
alluded to above and use these
results to prove Theorem~\ref{thm:main}; the reader who is interested
in these results and the proof may go directly to that section.\\

\noindent {\em Acknowledgement:}
We are grateful to E.~Ghate for some discussions regarding this article.

\section{Applications} \label{sec:app}
We continue to use the notation introduced in the previous section.

By a theorem of Mazur~(\cite[Theorem~III.5.1]{mazur:eisenstein}), if $r$ a prime bigger than~$7$,
then $r$ does not divide the order of the torsion subgroup of any elliptic curve over~$\Q$. 
So by Theorem~\ref{vatsal}, we get:
\begin{cor}
Suppose $A$ is an optimal semistable elliptic curve 
and $r$ a prime such that $r> 7$. 
Then $A[r]$ is irreducible. 
\end{cor}
This result is already known by~\cite[Theorem~4]{mazur:rational}; however,
the proof is different. In this context, we should also point out that
if one drops the hypothesis that $N$ is square-free, then 
there is a finite list of primes~$r$ such that an elliptic 
curve has a rational isogeny of degree~$r$: 
see~\cite[Theorem~1]{mazur:rational}.

As mentioned in~\cite[p.~307]{serre:propgal}, 
if $A$ is a semistable elliptic curve (it need not be optimal) and $r$ is a prime
such that $A[r]$ is reducible 
as a representation of~${\rm Gal}(\overline{\Q}/\Q)$, then 
this representation can be put in matrix form as
$$ 
\begin{bmatrix}
\chi' & * \\
0 & \chi'' 
\end{bmatrix},
$$
where $\chi'$ and $\chi''$ are characters of~${\rm Gal}(\overline{\Q}/\Q)$
such that one of them is the identity character and the other is the \mbox{$\bmod\ r$} cyclotomic character. If $\chi'$ is the identity
character, then $A$ has a nontrivial rational point of order~$r$; however, 
one does not know if $\chi'$ is the identity
character (as per loc. cit.). 
Theorem~\ref{vatsal} above implies that under the additional hypothesis that 
$A$ is optimal, one may take $\chi'$ to be the identity character.

In the rest of this section, we discuss an application of 
Theorems~\ref{vatsal} and~\ref{thm:main}
to the second part of 
the Birch and Swinnerton-Dyer (BSD)
conjecture for~$A$. 
Let $L_A(s)$ denote
the $L$-function of~$A$. 
Let $K_A$ denote the coefficient of the leading term of the Taylor
series expansion of~$L_A(s)$ at $s=1$, and let $R_A$ denote the 
regulator of~$A$.
Let $\Omega_A$ denote the volume of~$A({\R})$ calculated using
a generator
of the group of invariant differentials on the N\'{e}ron 
model of~$A$. 
Let $\Sha_{A}$ denote the Shafarevich-Tate group of~$A$, which
we assume is finite. 
If $p$ is a prime that divides~$N$, then
let $c_p(A)$ denote the order of the arithmetic component group
of~$A$ at~$p$ (also called the Tamagawa number of~$A$ at~$p$).
Then the {\it second part of the BSD~conjecture}
asserts the formula: 
\begin{eqnarray} \label{bsd}
\frac{K_A}{\Omega_A \cdot R_A}
\stackrel{?}{=} 
\frac{\Mid \Sha_{A} \miD \cdot \prod_p c_p(A)}{ \Mid A({\Q})_{\tor} \miD 
\cdot \Mid A^\vee({\Q})_{\tor} \miD }\ .
\end{eqnarray}
Based on numerical evidence, the first author has conjectured:
\begin{conj}\cite[Conjecture~2.4]{conjs} \label{conj2} 
Suppose that $A$ is an elliptic curve (recall that it is optimal by assumption).
If an odd prime~$\ell$ divides $c_p(A)$ for some prime~$p$ that divides~$N$, 
then
either $\ell$ divides $| A(\Q)_{\tor}|$
or the newform~$f$ 
is congruent to a newform of level dividing~$N/p$
(for all Fourier coefficients whose indices are coprime to~$N \ell$)
modulo a prime ideal over~$\ell$ in a number field containing
the Fourier coefficients of both newforms.
\end{conj}
This conjecture indicates some (conjectural) cancellation between the numerator
and denomintor of the right side of the BSD formula~(\ref{bsd}) above
(for more on such cancellations, see~\cite{conjs}).
Towards the conjecture above, we have the following mild generalization 
of~\cite[Proposition~2.3]{conjs}:
\begin{prop} \label{mild}
Let $\ell$ be an odd prime such that 
either $\ell \nmid N$ or for all primes $r$ that divide~$N$,
$\ell \nmid (r-1)$. 
If $\ell$ divides the order of the geometric component
group of~$A$ at~$p$ for some prime~$p || N$,
then either for some maximal ideal~$\m$ of~$\T$ of residue
characteristic~$\ell$ and containing~$I_f$, the representation
$\rho_{\m}$ is reducible, or the newform~$f$ 
is congruent to a newform of level dividing~$N/p$
(for all Fourier coefficients whose indices are coprime to~$N \ell$)
modulo a prime ideal over~$\ell$ in a number field containing
the Fourier coefficients of both newforms.
\end{prop}
\begin{proof}
The proof of Proposition~2.3 in~\cite{conjs} works with the following changes:
replace $E$ by~$A$ and $E[\ell]$ by~$\rho_{\m}$.
\end{proof}
Note that if $p$ is a prime that divides~$N$, and if a prime~$\ell$
divides~$c_p(A)$, then $\ell$ also divides
the order of the geometric component
group of~$A$ at~$p$.

In view of Theorems~\ref{vatsal} and~\ref{thm:main}, from the proposition above, we get:
\begin{prop} \label{prop0}
Let $\ell$ be a prime. Suppose $N$ is square-free and one of the following is true:\\
(a) $\ell \nmid 6N$. \\
(b)  $A$ is an ellitpic curve and 
either $\ell \nmid N$ or for all primes $r$ that divide~$N$, $\ell \nmid (r-1)$. \\
If $\ell$ divides the order of the geometric component
group of~$A$ at~$p$ for some prime~$p \divs N$,
then either $\ell$ divides $| A(\Q)_{\tor}|$
or the newform~$f$ 
is congruent to a newform of level dividing~$N/p$
(for all Fourier coefficients whose indices are coprime to~$N \ell$)
modulo a prime ideal over~$\ell$ in a number field containing
the Fourier coefficients of both newforms.
\end{prop}

The proposition above is a better result towards Conjecture~\ref{conj2}
than Proposition~\ref{mild}, at least 
when $\ell \ndiv 3N$.
Since, as mentioned earlier, if $r$ a prime bigger than~$7$ and $A$ is an elliptic curve,
then $r$ cannot not divide~$| A(\Q)_{\tor}|$, we get from the proposition above:
\begin{cor} Suppose $N$ is square-free and $A$ is an elliptic curve (recall 
that it is optimal by assumption).
Let $\ell$ be a prime such that $\ell > 7$. 
If $\ell$ divides the order of the geometric component
group of~$A$ at~$p$ for some prime~$p \divs N$, then
the newform~$f$ 
is congruent to a newform of level dividing~$N/p$
(for all Fourier coefficients whose indices are coprime to~$N \ell$)
modulo a prime ideal over~$\ell$ in a number field containing
the Fourier coefficients of both newforms.
\end{cor}

Finally, from Theorem~\ref{thm:main} and Proposition~\ref{prop0},  we have:
\begin{cor}
Suppose $N$ is square-free.
Let $\ell$ be a prime such that $\ell \nmid 6N$ and
$\ell$
does not divide~$|C|$.
If $\ell$ divides the order of the geometric component
group of~$A$ at~$p$ for some prime~$p \divs N$,
then 
the newform~$f$ 
is congruent to a newform of level dividing~$N/p$
(for all Fourier coefficients whose indices are coprime to~$N \ell$)
modulo a prime ideal over~$\ell$ in a number field containing
the Fourier coefficients of both newforms.
\end{cor}

\section{Some results on Fourier coefficients and
proof of the main theorem} \label{sec:lemmas}

We continue to use the notation introduced in Section~\ref{intro}.
Let $a_n = a_n(f)$ denote the $n$-th Fourier coefficient of~$f$.
Let $K_f$ denote the number field generated by the Fourier coefficients of~$f$.
We have a ring homomorphism $\T \ra K_f$ that takes $T_n$ to $a_n(f)$. Its
image is an order in~$K_f$, hence it is contained in the 
ring of integers of~$K_f$, which we denote by~$\Of$. 
We thus have an injection $\phi: \T_f := \T/I_f \ra \Of$. 
Recall that  $\m$ is a maximal ideal of~$\T$ that contains~$I_f$ and 
$r$ is the characteristic of~$\T/\m$. Let ${\m'}$ denote
the maximal ideal of~$\T_f$ corresponding to~$\m$.
Since $\Of$ is an integral extension of~$\T_f$, there
is a prime ideal~$\pp$ of~$\Of$ that lies over~${\m'}$, i.e.,
$\T_f \cap \phi^{-1}(\pp) = {\m'}$. In particular, $\phi({\m'}) \subseteq \pp$.
Since $\Of$ is Dedekind, $\pp$ is in fact a maximal ideal, but that will
not be needed in the sequel.
Recall that $\rho_{\m}$ 
denotes the canonical representation associated to~$\m$ (see~\cite[Prop.~5.1]{ribet:modreps}). 
Suppose that $\rho_{\m}$ is reducible.
We do not yet assume the hypotheses in our main theorem that 
$N$ is square-free or $r \ndiv 6N$. 
If $p$ is a prime that divides~$N$, then
let $w_p$ denote the sign of the Atkin-Lehner involution~$W_p$
acting on~$f$.

The following lemma is perhaps well known.
\begin{lem} \label{lem:mazur0}
For all primes $\ell \ndiv N$, we have
$a_\ell(f) \equiv 1 + \ell \bmod \pp$ and 
for all primes $p \divs N$, we have $a_p(f) = - w_p$.
\end{lem}
\begin{proof}
Suppose $\ell \ndiv N$. 
Since $\rho_{\m}$ is reducible, it follows from~\cite[p.~362]{yoo2} that 
$T_\ell - \ell -1 \in \m$. Thus the image of~$T_\ell - \ell -1$ under
the composite $\T/{\m} \ra \T_f/{\m'} \ra \Of/\pp$ is zero; but
this image is the coset represented by 
$a_\ell(f) - 1 - \ell$; hence $a_\ell(f) - 1 - \ell \in \pp$, which
proves the first assertion.

If $p \divs N$, then $a_p(f) = - w_p$ because
$U_p = - W_p$ on the new subspace of~$S_2(\Gamma_0(N),\C)$.
This finishes the proof of the lemma.
\end{proof}

The following proposition is well known, but we give details for the 
convenience of the reader.
\begin{prop}\label{prop:eis}
Suppose $N$ is square-free. 
For every prime~$p$ that divides~$N$,
suppose we are given
an integer $\delta_p \in \{1,p\}$ such that  
$\delta_p = 1$ for at least one~$p$.
Then there is an Eisenstein series~$E$ of weight~$2$ on~$\Gamma_0(N)$ 
which is an eigenfunction for all the Hecke operators such that
for all primes~$\ell \ndiv N$, we have 
$a_\ell(E) = \ell + 1$, and for all primes~$p \divs N$,
we have $a_p(E) = \delta_p$. 
Moreover, $E$ has integer Fourier coefficients, except perhaps~$a_0(E)$,
which is a rational number 
whose denominator (in reduced form) divides~$24$,
and:\\
(a) if for at least one prime~$p$ that divides~$N$ we have
$\delta_p = p$, then $a_0(E) =0$, and\\
(b) if $N$ is prime, then $a_0(E) = (1-N)/24$. 
\end{prop}
\begin{proof}
Everything except the last sentence is~\cite[Proposition~2.1]{tor}.
We now prove the last sentence. Note that $e(q)$ in 
 the first sentence in the proof of Proposition~2.1 in~\cite{tor}
has integer Fourier coefficients
except for the zeroth, which is $1/24$, and since $E$ is obtained
by ``raising levels'' from~$e(q)$, it follows that
$E$ has integer Fourier coefficients, except perhaps~$a_0(E)$,
which is a rational number 
whose denominator (in reduced form) divides~$24$.

We now prove part~(a). Suppose that $\delta_p = p$ for a prime~$p$ that divides~$N$. 
Note that in the proof of Proposition~2.1 in loc. cit.,
in the second last sentence of the second last paragraph, it
says that the $s$-th Fourier coefficient may be chosen to be~$1$ or~$s$;
this is done by using the construction in the second paragraph of the proof in loc. cit.,
with $g_r$ or $\tilde g_r$ respectively. We take $s$ in 
the second last paragraph of the proof of
Proposition~2.1 in loc. cit. to be our~$p$
for which $\delta_p = p$, and 
choose the $s$-th Fourier coefficient to be~$s$; 
so we are using the $\tilde g_r$ construction.
But $\tilde g_r(q) = g(q) - g(q^r)$, and so 
$a_0(\tilde g_r(q)) = a_0(g(q)) - a_0(g(q^r)) 
= a_0(g(q)) - a_0(g(q)) =0$, and so the zeroth Fourier coefficient 
at the end of the second last paragraph of the proof of
Proposition~2.1 in loc. cit. is~$0$.
Further level raising in the last paragraph of that proof does not change the 
zeroth Fourier coefficient, which proves part~(a) of the proposition.

Finally, we prove part~(b). Suppose that $N$ is prime.
Then by hypothesis, $\delta_N = 1$, and so 
we use
the construction~$g_r(q) = g(q) - r \cdot g(q^r)$, with $r=N$ and $g(q)=e(q)$. So
$a_0(E) = a_0(e_N) = (1-N)a_0(e) = (1-N)/24$,
which proves 
part~(b) of the proposition.
\end{proof}

Keeping in mind the strategy of the proof of 
our main theorem (Theorem~\ref{thm:main})
mentioned 
in the introduction, 
we see from the lemma and proposition above 
that coming up with an Eisenstein series~$E$ such that
$a_\ell(f) \equiv a_\ell(E) \bmod r$ for all primes $\ell \ndiv N$
is rather easy. Proving the congruence for all primes~$p$ that divide~$N$
for a suitable Eisenstein  series
is the tricky part, for which we need the results below.

\begin{lem}[Yoo] \label{lem:dum}
Suppose $r \geq 3$.
If $p$ is a prime such that $p \divs N$, $p \neq r$, and $w_p = 1$, then
$p \equiv -1 \bmod \pp$.
\end{lem}

\begin{proof} 
The argument is essentially given in the proof 
of~\cite[Lemma 2.1]{yoo2}; we repeat it here for the convenience of the reader.
Since $\rho_{\m}$ is reducible, 
$\rho_{\m} \isom {\bf 1} \oplus \chi_r$, where ${\bf 1}$ is the
trivial character and $\chi_r$ is the mod~$r$ cyclotomic character.
On the other hand, the semisimplification of the restriction of~$\rho_{\m}$
to~${\rm Gal}(\overline{\Q}_p/\Q_p)$ is isomorphic
to~$\epsilon \oplus \epsilon \chi_r$, where $\epsilon$ is the
unramified quadratic character with $\epsilon({\rm Frob})_p = a_p = -1$
because ${\m}$ is~$p$-new (cf.~\cite[Theorem~3.1(e)]{ddt}). 
From this, we get $p \equiv -1 \bmod r$, and since $r \in \pp$, we get our result. 
\end{proof}

\begin{lem}\label{lem:r}
Suppose $r \geq 3$ and $r \divs N$. Then $w_r = -1$. 
\end{lem}
\begin{proof}
As mentioned on p.~363 in~\cite{yoo2}, we have
$U_r \equiv 1 \bmod {\m}$ by~\cite[Lemma 1.1]{ribet:eisenstein}; it is
mentioned on p.~362 in~\cite{yoo2} that the quoted lemma also 
follows from the result by Deligne given in~\cite[Theorem~2.5]{edixhoven:weight}.
So if $w_r = 1$, then $a_r = -1$, and so
$-1 \equiv 1 \bmod {\m}$, i.e. $2 \in {\m}$, which
is not possible since $r$ is odd; thus $w_r = -1$. 
\end{proof}


\begin{prop} \label{prop:mazur2}
Recall that we are assuming that $\rho_{\m}$ is reducible.
Suppose that $N$ is square-free and
$r > 3$ or that $r=3$ and $\Of/\pp \isom \Z/3\Z$ (the latter happens if $f$ has
rational Fourier coefficients, i.e., $A$ is an elliptic curve). 
Then there is a prime~$p$ that divides~$N$ such that $w_p = -1$, i.e., $a_p(f)=1$.
\end{prop}

\begin{rmk} \label{rmk2}
Special cases of the proposition above have already appeared in the literature.
The result is stated for $r >3$ as Theorem~1.2(1) in~\cite{yoo:non-optimal}, and 
attributed to Ribet, but for a proof, the author refers to Theorem~2.6(ii)(b)
of~\cite{bd};  however in loc. cit., the result 
is stated  under the extra hypothesis
that $r \ndiv N$ (apart from $r \neq 3$).
The referee to an earlier version of this article informed us
that our proposition (at least when $r>3$ and $f$ has integer Fourier coefficients)
follows from Proposition~5.5 in~\cite{yoo2} or
Theorem 3.1.3 in~\cite{ohta:eisenstein}, though the desired result
is not stated as such in a simple form in either to be quoted 
(one would have to deduce it 
after reading the notations, etc.). In any case, 
Theorem 3.1.3 of~\cite{ohta:eisenstein} uses 
Proposition 3.4.2 in loc. cit, in whose proof the author
says that ``this inductive step was inspired by the similar
argument of Proposition 3.5'' of the preprint version of~\cite{tor}; 
in the publised version~\cite{tor}, this became Proposition 3.6, which
we are suitably modifying below.
\end{rmk}

\begin{proof}  
By Lemma~\ref{lem:r}, if $r \divs N$, then $w_r = -1$, and we are done.
Thus we may assume henceforth that
$r \ndiv N$. 

The rest of the proof for $r > 3$ is a generalization of the proof of Proposition~3.6
of~\cite{tor}. We give the full details for two reasons: one is that
the case of $r = 3$ requires referring to the case $r >3$,
and the second is that 
the proof in loc. cit. had an error that we have fixed here
(see Remark~\ref{rmk3} below).

\comment{
If $r > 3$, then the result follows from the proof of 
Proposition~3.6 of~\cite{tor}  with the following changes:
replace $\Z/r\Z$ by~$\Of/\pp$, \mbox{$\bmod\ r$} by~\mbox{$\bmod \pp$}, 
Lemma~3.1 by our Lemma~\ref{lem:mazur0}, 
Lemma~3.2 by our Lemma~\ref{lem:dum}, and Lemma~3.5 by~\cite[Lemma~2.1.1]{ohta:eisenstein}.
}

Suppose, contrary to the conclusion of the Proposition, 
that for every  prime~$p$ 
that divides~$N$, we have $w_p = 1$. Then by Lemma~\ref{lem:dum},
for every  prime~$p$ 
that divides~$N$, we have $p \equiv -1 \bmod \pp$ 
(note that $p \neq r$ since $r \ndiv N$).

Following \cite[p.~77 and p.~70]{mazur:eisenstein},
by a holomorphic modular form in~$\omega^{\tensor 2}$
on~$\Gamma_0(N)$ defined over a ring~$R$,
we mean a modular form in the sense of
\cite[\S1.3]{katz:antwerp350} (see also \cite[\S~VII.3]{deligne-rapoport}).
Thus such an object is a rule which assigns to each pair
$(E_{/T},H)$, where $E$ is an elliptic curve over an $R$-scheme~$T$
and $H$ is a finite flat subgroup scheme of~$E_{/T}$ of order~$N$,
a section of~$\omega_{\scriptscriptstyle E_{/T}}^{\tensor 2}$,
where $\omega_{\scriptscriptstyle E_{/T}}$ is the sheaf of invariant differentials.
Since $r \ndiv N$, 
by~\cite[Lemma~1.3.5]{ohta:eisenstein}, 
if $g$ is a modular form of weight~$2$
on~$\Gamma_0(N)$ with coefficients in~$\Of/\pp$,
then there is a holomorphic modular form  in~$\omega^{\tensor 2}$
on~$\Gamma_0(N)$ defined over~$\Of/\pp$, which we will
denote $g \bmod \pp$,  such that the
$q$-expansion of~$g \bmod \pp$ agrees with the $q$-expansion
of~$g$ modulo~$\pp$. 

If $M$ is a postive integer, then let
us say that a holomorphic modular form~$g$ in~$\omega^{\tensor 2}$
on~$\Gamma_0(M)$ defined over~$\Of/\pp$ is {\em special at level~$M$}
if $a_n(g) \equiv \sigma(\frac{n}{(n,M)}) \prod_{p \mid M} 
(-1)^{{\rm ord}_p(n)} \bmod \pp$
for all positive integers~$n$.
Using Lemma~\ref{lem:mazur0} and the fact that $f$ is an eigenvector
for all the Hecke operators,
we see that $f \bmod \pp$ is special at level~$N$. 

Consider first the case where $r > 3$.

\noindent{\em Claim:} 
If $M$ is a square-free integer and 
$g$ is a holomorphic modular form in~$\omega^{\tensor 2}$
on~$\Gamma_0(M)$ 
defined over~$\Of/\pp$ that is special at level~$M$ and 
$s$ is a prime that divides~$M$, then there exists
a holomorphic modular form in~$\omega^{\tensor 2}$
on~$\Gamma_0(M/s)$ 
defined over~$\Of/\pp$ that is special at level~$M/s$
(which is also square-free).
\begin{proof}
By Proposition~\ref{prop:eis},
there is an Eisenstein series~$E$ which is an eigenvector
for all the Hecke operators, with $a_\ell(E) = \ell + 1$
for all primes $\ell \ndiv M$,
$a_p(E) = p$ 
for all primes~$p$ that divide~$M$ except $p = s$, and $a_s(E) = 1$; also
since $r > 3$, we can talk about $E \bmod \pp$.
Let $p_1, \ldots, p_t$ be the distinct primes that divide~$M/s$.
Then for any positive integer~$n$,
$$a_n(E) = \sigma \bigg(\frac{n}{(n,M)} \bigg) 
\prod_{i=1}^{t}{{p_i}^{{\rm ord}_{p_i}(n)}}\ .$$
Since
$p_i \equiv -1 \bmod \pp$ for $i = 1, \ldots, t$,
we see that $a_n( E) \equiv a_n( g) \bmod \pp$ if $n$ is coprime to~$s$, and thus 
$(E(q)-g(q)) \bmod \pp$ 
is a power series in $q^s$, i.e., 
there is an $h(q) \in (\Of/\pp)[[q]]$
with $ h(q^s)$ equal to $(E(q)-g(q)) \bmod \pp$.
By~\cite[Lemma~2.1.1]{ohta:eisenstein}, $h(q)$ is the $q$-expansion
of a 
holomorphic modular form, which we again denote~$h$, in~$\omega^{\tensor 2}$
on~$\Gamma_0(M/s)$ defined over~$\Of/\pp$.

Let $g' = h/2$.
We shall now show that $g'$ is special of level~$M/s$.
Let $n$ be a positive integer, $m' = \frac{n}{(n,s)}$,
and $e = {\rm ord}_s(n)$ (so $n = m' s^e$). 
Then considering that $m'$ is coprime to~$s$, we have
$a_n(E) = a_{m'}(E) a_{s^{e+1}}(E)$ since $E$ is an eigenfunction and 
$a_n(g) \equiv a_{m'}(g) a_{s^{e+1}}(g) \bmod \pp$ since $g$ is special. So
\begin{eqnarray}\label{eqn1}
a_n(h)&  = & a_n(E) - a_n(g) \nonumber\\
& \equiv &
a_{m'}(E) a_{s^{e+1}}(E) - a_{m'}(g) a_{s^{e+1}}(g) \bmod \pp \nonumber\\
& \equiv & a_{m'}(g) (a_{s^{e+1}}(E) - a_{s^{e+1}}(g))\bmod \pp,
\end{eqnarray}
where the last congruence follows since $a_{m'}(g) \equiv a_{m'}(E) \bmod \pp$,
considering that $m'$ is coprime to~$s$.

	From the definitions and the fact that \(s \mid \mid M\),  we find that
	\[
	a_{s^{e+1}}(E) = \sigma \bigg( {\frac{s^{e+1}}{s}} \bigg) \prod_{p \mid M, p\neq s} p^{\operatorname{ord}_{p}(s^{e+1})} = \sigma(s^{e})
	\]
	and
	\[
	a_{s^{e+1}}(g) \equiv \sigma \bigg( {\frac{s^{e+1}}{s}}\bigg) \prod_{p \mid M} (-1)^{\operatorname{ord}_{p}(s^{e+1})} \equiv  \sigma(s^{e})(-1)^{e+1} \bmod \pp. 
	\]

So we have
\begin{eqnarray}\label{eqn2}
a_{s^{e+1}}(E) - a_{s^{e+1}}(g) 
 \equiv 
 (1-(-1)^{e+1})\sigma(s^{e}) 
\bmod \pp .
\end{eqnarray}
We claim that 
$(1-(-1)^{e+1})\sigma(s^{e}) = 2 \sigma(s^{e})$: if 
$e \equiv 1 \bmod 2$, then $\sigma(s^e) = 0$, so the equality holds trivially,
and if $e \equiv 0 \bmod 2$, then $1-(-1)^{e+1} = 2$, so again the equality holds.
Putting this proven claim in~(\ref{eqn2}) and the result 
in~(\ref{eqn1}),
we get
\begin{eqnarray}\label{eqn4}
a_n(h) \equiv  a_{m'}(g) \cdot 2\sigma(s^e) 
\equiv 2 \sigma\bigg(\frac{m'}{(m',M)}\bigg)  \prod_{p \mid M} (-1)^{{\rm ord}_p(m')}  
\cdot \sigma(s^e) 
\bmod \pp, 
\end{eqnarray}
where the last congruence follows since $g$ is special at level~$M$.
Now since $n = m' s^e$, with $m'$ coprime to~$s$ and $s \ndiv (M/s)$, we have
\begin{eqnarray}\label{eqn5}
\sigma\bigg(\frac{m'}{(m',M)}\bigg)  \sigma(s^e)  
= \sigma\bigg(\frac{m' s^e}{(m',M)}\bigg) 
= \sigma\bigg(\frac{m' s^e}{(m' s^e,M/s)}\bigg) 
= \sigma\bigg(\frac{n}{(n,M/s)}\bigg) 
\end{eqnarray}
and 
\begin{eqnarray}\label{eqn6}
\prod_{p \mid M} (-1)^{{\rm ord}_p(m')} 
=  \prod_{p \mid M, \ p \neq s} (-1)^{{\rm ord}_p(m's^e)} 
= \prod_{p \mid (M/s)} (-1)^{{\rm ord}_p(n)}.
\end{eqnarray}
Using~(\ref{eqn5}) and~(\ref{eqn6}) in~(\ref{eqn4}), and
recalling that $g' = h/2$, we see that
\begin{eqnarray*}
a_n(g') \equiv 
\sigma\bigg(\frac{n}{(n,M/s)}\bigg) \prod_{p \mid (M/s)} (-1)^{{\rm ord}_p(n)}
\bmod \pp,
\end{eqnarray*}
i.e., $g'$ is special of level~$M/s$, as claimed.
\end{proof}

Starting with $f \bmod \pp$ (note that $r \ndiv N$), and repeatedly using the claim,
we see that there is 
a holomorphic modular form that is special of level~$1$, which is nontrivial
since the coefficient of~$q$ is~$1 \bmod \pp$ for a special form
(of any level). 
But by~\cite[Lemma II.5.6(a)]{mazur:eisenstein},
there are no nontrivial holomorphic modular forms
of level~$1$ in $\omega^{\tensor 2}$ defined over 
a field of characteristic other than~$2$ and~$3$.
This contradiction proves the case where $r > 3$.

\comment{
The strategy of the proof 
Proposition~3.6 of~\cite{tor} is that 
if the conclusion of that proposition is not true, then 
one can lower the level of $f \bmod \pp$ 
(the notation is as in loc. cit.)  using
a suitable Eisenstein series $E \bmod \pp$ obtained from 
Proposition~\ref{prop:eis} one prime at a time to
get a certain \mbox{$\bmod\ r$} form of level~$1$ in whose
$q$-expansion, the coefficient of~$q$ is~$1$,
which contradicts \cite[Lemma II.5.6(a)]{mazur:eisenstein}, which says that 
there are no nontrivial \mbox{$\bmod\ r$} forms
of level~$1$ defined over a field of characteristic other than~$2$ and~$3$.
}

Suppose now that $r=3$ and $\Of/\pp \isom \Z/3\Z$.
In the proof of the claim above, 
if $M$ is not a prime, then 
we have at least one prime $p \divs M$
where $a_p(E) = p$, and so by 
by part~(a) of Proposition~\ref{prop:eis}, we 
can talk of $E \bmod  \pp = E \bmod 3$ even though $r \not> 3$, and the 
conclusion of the claim holds. 
So if $N$ is not prime,
then starting with $f \bmod \pp$ (note that $r \ndiv N$), and repeatedly using the claim,
we see that there is 
a holomorphic modular form that is special of level a prime, call it~$p$. 
If $N$ were prime to start with, then we can take~$p=N$ to again
have a holomorphic modular form that is special of level the prime~$p$.

At this point, 
we follow a strategy that goes back to Mazur (see the proof of Prop.~II.14.1 on p.~114
in~\cite{mazur:eisenstein}), part of which is used in
the proof of Propostion 3.4.2 in~\cite{ohta:eisenstein}. For
the sake of being clear, rather than simply referring
to the sources above for the rest of the proof, we give the full details in our context.

Since $\Of/\pp \isom \Z/3\Z$ by hypothesis, 
$f \bmod \pp$ is a form defined over~$\Z/3\Z$. 
Suppose first that $p \equiv 1 \bmod 3$.
Then we can consider $E \bmod 3$ by part~(b) of Proposition~\ref{prop:eis}, 
and lower the level again to get a form of level~$1$ 
for whom the coefficient of~$q$ is~$1$.  The
constant term of this form is~$a_0(E)$ modulo~$3$, and
$a_0(E) = (p-1)/24$. If $p \not\equiv 1 \bmod 9$, then $a_0(E)$ is a unit \mbox{$\bmod\ 3$},
and by multiplying the form above by the inverse of this 
unit, we get
a \mbox{$\bmod\ 3$} form with constant term~$1$ and for which the coefficient of~$q$ is non-zero, which
contradicts \cite[Prop~II.5.6(c)(ii)]{mazur:eisenstein}.
If $p \equiv 1 \bmod 9$, then $a_0(E)$ is zero modulo~$3$, and this contradicts
\cite[Prop~II.4.10]{mazur:eisenstein}.

Finally, we are left with the case $p \equiv 2 \bmod 3$ (recall 
that we are assuming that $ r \ndiv N$, so $r \neq p$, 
and $r =3$, so $p\neq 3$). 
We consider the form $h= 3 E - \text{``}\times 3\text{''} f$ 
defined over~$\Z/9\Z$, with notation as in Lemma~3.3.2 in~\cite{ohta:eisenstein}
(where the construction $\text{``}\times 3\text{''}$ is attributed to~\cite[p.~86]{mazur:eisenstein}).
Since $a_0(f) = 0$,
we have $a_0(h) =  a_0(3E) = (1 - p)/8$ and since $p \equiv 2 \bmod 3$, this 
is a unit \mbox{$\bmod\ 9$}. 
Multiplying~$h$ by the inverse of this unit, we get 
a \mbox{$\bmod\ 9$} form  with $q$-expansion beginning with~$1$ 
and for which the coefficient of~$q$ is non-zero, which 
contradicts \cite[Lemma II.5.6(b)]{mazur:eisenstein} as $9  \ndiv 12$. 
\end{proof}

\begin{rmk}   \label{rmk3}
We could like to take the opportunity to point out a correction 
in the  proof of 
Proposition~3.6 of~\cite{tor}: in equation~(3.6) in loc. cit., the second line is 
not justified, and a priori need not be true. The proposition above
generalizes Proposition~3.6 of loc. cit., and so the reader 
should see the proof above instead of the one given in loc. cit.
\end{rmk}
\comment{
Using SAGE, we found that for any optimal elliptic curve~$E$ in 
Cremona's database~\cite{cremona:onlinetables} of elliptic curves (which includes all
curves with conductor less than~$130,000$), 
if $N$ is square-free and the odd part of~$E(\Q)_{\tor}$ is nontrivial, 
then $w_p = -1$ for at least one prime~$p$ that divides~$N$
(which indicates that the hypothesis that $r$ not divide~$3N$ in Proposition~\ref{prop:mazur2}
may be redundant).
The statement in the previous sentence is not true without the hypothesis
that $N$ is square-free. 
For example, 
the optimal elliptic curve~490a1 (note that $490= 2 \cdot 5 \cdot 7^2$) has a $3$-torsion point,
but $w_p = 1$ for all $p$ dividing~$490$. At the same time,
for all optimal elliptic curves in Cremona's database (of arbitrary conductor)
that have a $5$ or~$7$-torsion point, there is a prime~$p$ that divides the conductor
such that  $w_p=-1$. 

}

\begin{prop}\label{prop:red}
Suppose that $r \ndiv 2N$. If $A[\m]$ is reducible, then so is 
$\rho_{\m}$. 
\end{prop}
\begin{proof} We prove the contrapositive. 
So suppose $\rho_{\m}$ is irreducible. Then by~\cite[Prop.~5.2(b)]{ribet:modreps}, 
$J_0(N)[\m]$ is also irreducible. But $A[\m]$ is a subrepresentation 
of $J_0(N)[\m]$. So $A[\m]$ is either trivial or all of~$J_0(N)[\m]$; in either case, it
is irreducible, as was to be shown.
\end{proof}

\begin{proof}[Proof of Theorem~\ref{thm:main}]
Recall that we are assuming that $N$ is square-free, $r \ndiv 6N$, and 
either $\rho_{\m}$ or~$A[\m]$ is reducible. 
By Proposition~\ref{prop:red}, in any case, $\rho_{\m}$ is reducible. Thus all the results
in this section before Proposition~\ref{prop:red} apply.
We modify the proof of Theorem~1.2 of~\cite{tor} in Section~3 of loc. cit., by
replacing \mbox{$\bmod\ r$} by~\mbox{$\bmod\ \pp$}, Proposition~3.6 by our Proposition~\ref{prop:mazur2},
Lemma~3.2 by our Lemma~\ref{lem:dum}, 
Lemma~3.1 by our Lemma~\ref{lem:mazur0}, $\Z/r\Z$ by~$\Of/\pp$, Lemma~3.5  
by~\cite[Lemma~2.1.1]{ohta:eisenstein}, and
``$f$ is ordinary at~$r$'' by ``$f$ is ordinary at~$\pp$'', to conclude that~$A[\m] \cap C_E$
is nontrivial. The remaining conclusions of the theorem follow considering that
$C_E$ is rational, as $N$ is square-free. 
\end{proof}

\bibliographystyle{amsalpha}         

\bibliography{biblio}
\end{document}